\documentclass[12pt,bezier]{article}
\usepackage{amsmath, amsthm, amssymb}

\newtheorem{theorem}{Theorem}

\newtheorem{corollary}[theorem]{Corollary}
\newtheorem{lemma}[theorem]{Lemma}
\newtheorem{remark}[theorem]{Remark}
\newtheorem{conjecture}[theorem]{Conjecture}
\newtheorem{definition}[theorem]{Definition}

\newcommand{\1}{{\bf 1}}
\newcommand{\y}{{\bf y}}
\newcommand{\x}{{\bf x}}
\newcommand{\z}{{\bf z}}
\newcommand{\uu}{{\bf u}}
\newcommand{\w}{{\bf w}}
\newcommand{\s}{{\cal S}}

\newcommand{\e}{\epsilon}

\title{Signed graphs with maximal index}

\author{ Ebrahim Ghorbani$^{\,\rm a,b,}$\thanks{Corresponding author, e\_ghorbani@ipm.ir}
\qquad 	Arezoo Majidi$^{\,\rm a}$   \\[.4cm]
	{\sl $^{\rm a}$Department of Mathematics, K. N. Toosi University of Technology,}\\
	{\sl P. O. Box 16765-3381, Tehran, Iran}\\
	{\sl $^{\rm b}$Department of Mathematics, University of Hamburg, }\\
	{\sl Bundesstra\ss e 55 (Geomatikum), 20146 Hamburg, Germany }}

\begin{document}

\maketitle

\begin{abstract}
	The index of a signed graph is the largest eigenvalue of its adjacency matrix.
For positive integers $n$ and $m\le n^2/4$, we determine the maximum index of complete signed graphs with $n$ vertices and $m$ negative edges and characterize
the signed graphs achieving this maximum. This settles (the corrected version of) a conjecture by  Koledin and Stani\'c (2017).

		\vspace{5mm}
	\noindent {\bf AMS Mathematics Subject Classification\,(2010):} 05C50, 05C22
\end{abstract}

\section{Introduction}

A {\em signed graph} $\Gamma=(G,\sigma)$ consists of a simple graph $G$ together with a function $\sigma$ assigning a $+1$ or $-1$ to each edge of $G$. The (unsigned) graph $G$ is said to be the {\em underlying graph} of $\Gamma$, while the function $\sigma$
is called the {\em signature} of $\Gamma$.
For a simple graph $G$ with vertex set $\{v_1, \ldots ,v_n \}$, the {\em adjacency matrix} $A(G)=(a_{ij})$ is an $n \times n$ symmetric matrix with $a_{ij}=1$ if $v_i$ and $v_j$ are adjacent, and $a_{ij}=0$ otherwise.
In signed graphs, edge signs are usually interpreted as $\pm1$. In this way, the
adjacency matrix $A(\Gamma)$ is naturally defined following that of unsigned graphs, that is
by putting $+1$ or $-1$ whenever the corresponding edge is either positive or negative, respectively.
As $A(\Gamma)$ is a real symmetric matrix, its eigenvalues are all real numbers.
The {\em index} of the signed graph $\Gamma$ is the largest eigenvalue of $A(\Gamma)$.
The {\em spectral radius} of $\Gamma$ is the largest absolute value of the eigenvalues of $A(\Gamma)$.
These two coincide when the absolute values of the eigenvalues of $A(\Gamma)$ do not exceed its index. This is the case for unsigned graphs by the Perron--Frobenius theorem (see  \cite[Theorem 2.2.1]{bh}), a property that does not hold in general for signed graphs.

Several questions about signed graphs with extremal spectral radius  have been studied in the literature.
A natural question is to identify which signature leads to the minimum
spectral radius \cite{bckw}. This problem has important connections and consequences in the theory of expander graphs. Bilu and Linial \cite{bl} conjectured that every connected $d$-regular graph has a signature
with spectral radius at most $2\sqrt{d-1}$. If true, this conjecture would imply the existence of an infinite family of $d$-regular Ramanujan graphs.  Marcus, Spielman
and Srivastava~\cite{mss} proved the Bilu--Linial conjecture for bipartite graphs.
A similar problem for the $n$-dimensional hypercubes $Q_n$ is also of particular interest.
As $Q_n$ is an $n$-regular graph with $2^n$ vertices, for any signature $\sigma$, the
sum of the squares of the eigenvalues of $A(Q_n,\sigma)$ is equal to  trace$(A(Q_n,\sigma)^2)=n2^n$. It follows that spectral radius of $(Q_n,\sigma)$ is at least $\sqrt n$.
Recently, Huang \cite{h} constructed a signed adjacency matrix of $Q_n$
with spectral radius $\sqrt n$,  from which he concluded that every induced subgraph of $Q_n$ on more than $2^{n-1}$ vertices has maximum degree at least $\sqrt n$. This led to a breakthrough proof of the Sensitivity Conjecture from
theoretical computer science.

In this paper, we deal with signed graphs with maximal index.
To be more precise, we consider the problem of identifying the signed graphs with maximal index among the complete signed graphs with a fixed number of vertices and number of negative edges.
 This problem was initiated in \cite{ks}, where the following conjecture was posed. Here, as usual, the notation  $K_{r,n-r}$ denotes the complete bipartite graph with parts consisting of $r$ and $n-r$ vertices.

\begin{conjecture}[Koledin and Stani\'c \cite{ks}] \label{conj} The complete signed graph with $n$ vertices
	and $m\le\lfloor n^2/4\rfloor$ negative edges that maximizes the index is as follows:
	\begin{itemize}
		\item[\rm(i)] If $m < n - 1$, then negative edges induce the star $K_{1,m}$.
		\item[\rm(ii)] Otherwise, let $r$ with $r\le \lfloor n/2\rfloor$ be the largest integer that satisfies $r(n-r)\le m$.
		\begin{itemize}
			\item[\rm(ii.a)] If $r(n-r) = m$, then negative edges induce the complete bipartite graph
			$K_{r,n-r}$.
			\item[\rm(ii.b)] Otherwise, negative edges induce a bipartite graph with $r + 1$ vertices in
			one and $n-r-1$ vertices in the other part so that all but one vertex in the first
			part are adjacent to all vertices in the other.
		\end{itemize}
	\end{itemize}
\end{conjecture}

The main purpose of this paper is to prove Conjecture~\ref{conj}.
However, as we shall see in Theorem~\ref{thm:main} below, the part (ii.b) of the conjecture is not correct as stated.
In fact, in the case that $r(n-r)< m<(r+1)(n-r-1)$ where
\begin{equation}\label{eq:m}
(r+1)(n-r-1)-m>m-r(n-r),
\end{equation}
 the complete signed graph with maximal index is different from the one predicted in (ii.b).

For an unsigned graph $G$, as usual we use $V(G)$ and $E(G)$ to denote the vertex set and edge set of $G$, respectively. $|V(G)|$ and $|E(G)|$ are called the {\em order} and the {\em size} of $G$, respectively.

\begin{definition}\label{def:Hm}\rm
Let $n$ be a positive integer and $m\le\lfloor n^2/4\rfloor$.
Let $d(n-d)$ with $d\le\lfloor n/2\rfloor$ be the closest integer among
 $$0,\,n-1,\,2(n-2),\,3(n-3),\ldots,\lfloor n/2\rfloor\lceil n/2\rceil$$
 to $m$, and
$t:=|m-d(n-d)|$. We define a graph $H_{n,m}$ as follows.
If $m< d(n-d)$, $H_{n,m}$ is a graph obtained  by removing the edges of an star $K_{1,t}$ from $K_{d,n-d}$.
If $m\ge d(n-d)$, $H_{n,m}$  is a graph obtained by adding the edges of an star $K_{1,t}$ into one of the parts of $K_{d,n-d}$.
\end{definition}

In particular, if $m=d(n-d)$, then $H_{n,m}=K_{d,n-d}$, and if $m<n-1$, $H_{n,m}=K_{1,m}\cup(n-m-1)K_1$. In the remaining cases, up to isomorphism, there are
two choices for $H_{n,m}$, and either of them are referred to as $H_{n,m}$.

If the negative edges of a complete signed graph $\Gamma$ induce the (unsigned) graph $H$, we also use the notation $(K_n,H)$ to specify $\Gamma$.

Here is the main result of the paper.
\begin{theorem}\label{thm:main} Among the complete signed graphs with $n$ vertices
	and $m$ negative edges, $(K_n,H)$ has the maximum index if and only if $H$ is isomorphic to a $H_{n,m}$.
\end{theorem}

Theorem~\ref{thm:main} settles (the corrected version of) Conjecture~\ref{conj}. Note that, the parameter $d$ of Definition~\ref{def:Hm} is equal to either  $r$ of Conjecture~\ref{conj} or to $r+1$.  When $d=r+1$, then the maximal graph suggested in Conjecture~\ref{conj}  is $(K_n,H_{n,m})$. However, if
$d=r$ which is the case when \eqref{eq:m} holds, then the maximal graph predicted in Conjecture~\ref{conj}  is different from $(K_n,H_{n,m})$ and it has smaller index than $(K_n,H_{n,m})$.

  We remark that, the special case of Conjecture~\ref{conj} when the negative edges induce a tree was proved by Akbari {\em et al.} \cite{adhm}.

As a corollary, we will obtain the following quantified version of Theorem~\ref{thm:main}.

\begin{corollary}\label{cor:index}
	Let $n$ be an integer and $m\le\lfloor n^{2}/4\rfloor$. The largest index of complete signed graphs with $n$ vertices and $m$ negative edges is equal to $n-1-\xi,$ where $\xi$ is the smallest real satisfying
	$\xi(n-\xi)^2=4t(n-1-t)$ with $t:=\min_{0\le j\le \lfloor n/2\rfloor}|m-j(n-j)|$.
	In particular,
	\begin{equation}\label{eq:<xi<}
	\frac{4t(n-1-t)}{n^2}\le\xi\le \frac{4t(n-1-t)}{(n-1)^2}.
	\end{equation}
\end{corollary}

The rest of the paper is organized as follows. 
We determine the index of the signed graph $(K_n,H_{n,m})$ in Section~\ref{sec:maxindx}.
Theorem~\ref{thm:main} will be proved in Section~\ref{sec:maxgraph}.

\section{The maximal index}\label{sec:maxindx}

In this section, we determine the index of the signed graph $(K_n,H_{n,m})$.
In the next section, we will prove that $(K_n,H_{n,m})$ has the largest index among complete signed graphs with $n$ vertices and $m$ negative edges which establishes the maximality of the index of $(K_n,H_{n,m})$.

Let $H$ be a simple graph with vertex set $V=V(H)=\{v_1,\ldots,v_n\}$.  The {\em Seidel matrix} $\s(H)=(s_{ij})$ of $H$ is an $n\times n$ matrix where $s_{11}=\cdots=s_{nn}=0$ and for $i\ne j$,
$s_{ij}$ is $-1$ if $v_i$ and $v_j$ are adjacent, and is
$1$ otherwise.  If $\Gamma=(K_n,H)$, that is the complete signed graph whose negative edges induce the unsigned graph $H$, then $$A(\Gamma)=J-I-2A(H)=\s(H),$$
in which $J$ and $I$ are the all ones and the identity matrices of order $n$, respectively.
So the adjacency matrix of $\Gamma$  coincides with the Seidel matrix  of $H$. Therefore, the index of $\Gamma$ is the same as the largest eigenvalue of $\s(H)$ which
we denote it by $\rho(H)$.
Occasionally, we also call $\rho(H)$ the {\em index} of $\s(H)$.

Let $U$ be a subset of $V$ and $U'=V\setminus U$. The {\em Seidel switching} on $H$ with respect to $U$ leaves the subgraphs induced by $U$ and $U'$ unchanged, but deletes all edges between $U$ and $U'$,
and inserts all edges between $U$ and $U'$ that were not present in $H$. Thus, if
$$\s(H) =\bordermatrix{~ & U & U'\cr U& A_1& A_2\cr U'& A_2^\top& A_3},$$
and $H'$ is the resulting graph, then
$$\s(H') =\begin{pmatrix}A_1& -A_2\\-A_2^\top& A_3\end{pmatrix}.$$
The matrices $\s(H)$ and $\s(H')$ are similar, and thus have the same eigenvalues.
The graph $H$ is said to be {\em switching equivalent} with $H'$.

For $0\le m\le n-1$, we denote the graph $K_{1,m}\cup(n-m-1)K_1$ by $S_{n,m}$. In particular, $S_{n,0}=\overline K_n$, the graph with no edges.

\begin{lemma}\label{lem:index}
		Let $m\le\lfloor n^{2}/4\rfloor$ and $t:=\min_{0\le j\le \lfloor n/2\rfloor}|m-j(n-j)|$. Then
	$\rho(H_{n,m})=\rho(S_{n,t})$. In particular, if $m=d(n-d)$ for some integer $d$, then $\rho(H_{n,m})=n-1$.
	\end{lemma}
\begin{proof}
	By Definition \ref{def:Hm}, $H_{n,m}$ is obtained from a $K_{d,n-d}$ by adding or removing the edges of a $K_{1,t}$. Let $U$ be either of the parts of the above $K_{d,n-d}$. Then, $S_{n,t}$ can be obtained from $H_{n,m}$ by the Seidel switching with respect to
$U$.
This shows that $H_{n,m}$ and $S_{n,t}$ are switching equivalent, and thus
	$\rho(H_{n,m})=\rho(S_{n,t})$.
 If $m=d(n-d)$, then $t=0$, and so,  $\rho(H_{n,m})=\rho(\overline K_n)=n-1$ since $\s(\overline K_n)=J-I$.
\end{proof}	

In passing we remark that for any graph $H$ of order $n$,  $\rho(H)\le n-1=\rho(\overline K_n)$.
Moreover, it is straightforward to verify that if $\rho(H)=n-1$, then $H$ must be switching equivalent with $\overline K_n$.
So, from the particular case $m=d(n-d)$ of Lemma~\ref{lem:index}, the part~(ii.a) of Conjecture~\ref{conj} follows.

\begin{theorem}\label{thm:index}
	Let $n\ge3$, $m\le\lfloor n^{2}/4\rfloor$, and $t=\min_{0\le j\le \lfloor n/2\rfloor}|m-j(n-j)|$. Then
	$\rho(H_{n,m})=n-1-\xi,$ where $\xi$ is the smallest real satisfying
	$\xi(n-\xi)^2=4t(n-1-t)$.
	In particular,
	\begin{equation}\label{eq:<xi<}
	\frac{4t(n-1-t)}{n^2}\le\xi\le \frac{4t(n-1-t)}{(n-1)^2}.
	\end{equation}
\end{theorem}
\begin{proof}
By Lemma~\ref{lem:index}, we only need to determine $\rho(S_{n,t})$.	
Note that $t\le(n-1)/2$.
If $t=0$, then $\rho(S_{n,t})=n-1$, as required.
Therefore, we assume that $t\ge1$.
The vertices of $S_{n,t}$ have degrees $t,1$, and $0$. The partition of $V(S_{n,t})$ according to these degrees gives rise to an equitable partition  (see \cite[p.~24]{bh}) for $\s(S_{n,t})$ with the quotient matrix
\begin{equation}\label{eq:Q}
Q:=\begin{pmatrix}
	0& -t& n-t- 1\\-1& t- 1& n - t - 1\\ 1& t& n -t - 2
	\end{pmatrix}.
\end{equation}
	The characteristic polynomial of $Q$ is the cubic polynomial
		$$f(x):=x^3 + (3-n)x^2 + (3-2n)x - 4t^2 + 4nt - 4t - n + 1.$$
		If we remove the central vertex of $S_{n,t}$, we obtain $\overline K_{n-1}$, with $\s(\overline K_{n-1})$ having eigenvalue $-1$ with multiplicity $n-2$. Therefore, by interlacing (see \cite[Corollary 2.5.2]{bh}), $\s(S_{n,t})$ has the eigenvalue $-1$ with multiplicity at least $n-3$.
	The polynomial $f$ has no zero $x=-1$ (in fact, $-1$ is a zero of $f$ if and only if $t=0$ or $t=n-1$ which is not the case).
	It follows that all the eigenvalues of $\s(S_{n,t})$ are the zeros of $f$ together with $-1$ with multiplicity $n-3$.
It turns out that the largest zero of $f$ is the index of $\s(S_{n,t})$. Let
\begin{equation}\label{eq:g}
g(y):=f(n-1-y)=4t(n-1-t)-y(n-y)^2.
\end{equation}
	Then $\rho(S_{n,t})=n-1-\xi$ where $\xi$ is the smallest zero of $g$.
	
	To show \eqref{eq:<xi<}, let $\xi_1$ and $\xi_2$ be the lower and the upper bounds in \eqref{eq:<xi<}, respectively.  We observe that
	$$g(\xi_1)=\frac{64}{n^6}\left(t^2 + t(1-n) +\frac{n^3}{2}\right)(t-n + 1)^2t^2\ge0,$$
	and
	$$g(\xi_2)=\frac{4}{(n - 1)^6}\Big(4t^2+(n-1)\big((2n-1)(n-1)-4t\big)\Big)(2t-n + 1)^2(t-n+1)t\le0.$$
	It follows that $g$ has a zero in the interval $[\xi_1,\xi_2]$.
	Note that the derivative of $g$ with respect to $y$ is $(n-y)(3y-n)$ which has zeros at $n$ and $n/3$.
	This means that $g$ has a zero in the interval $(n/3,n)$ and a zero greater than $n$. As $\xi_2\le1\le n/3$, we find
 that the zero of $g$ lying in $[\xi_1,\xi_2]$ is indeed the smallest zero of $g$. This completes the proof.
\end{proof}

\section{The graphs with maximal index}\label{sec:maxgraph}

This section is devoted to the proof of Theorem~\ref{thm:main}. This together with Theorem \ref{thm:index} will also imply Corollary \ref{cor:index}. Since the adjacency matrix of the signed graph $(K_n,H)$ is the same as the Seidel matrix of the unsigned graph $H$, in this section we only deal with unsigned graphs and their Seidel matrices.

We start with a lemma which shows that,  except one special case, the eigenvector for the maximal index of graphs with size $\le n^2/4$ has no zero components.
\begin{lemma}\label{lem:nonzero}
 Let $H$ be a graph such that $\s(H)$ has the largest index among the graphs of order $n$ and size $m\le\lfloor n^2/4\rfloor$.
 \begin{itemize}
 	\item[\rm(i)] If $m\ne(n-1)/2$, then $\rho(H)>n-2$ and any eigenvector corresponding to $\rho(H)$ has no zero components.
 	\item[\rm(ii)] If $m=(n-1)/2$, $n\ge4$, and  $\rho(H)$ has an eigenvector with a zero component, then $\rho(H)=n-2$ and $H$ is isomorphic with $S_{n,(n-1)/2}$.
 \end{itemize}

\end{lemma}
\begin{proof}{  Let $\x=(x_1,\ldots,x_n)$ be a unit eigenvector corresponding to the index of $\s(H)$.
		With the notation of the proof Theorem~\ref{thm:index}, for some $0\le t\le (n-1)/2$, $\rho(H_{n,m})=\rho(S_{n,t})\ge n-1-\xi_2$.
		
		(i) If $t=0$, then $\rho(S_{n,t})=n-1$. If $0<t<(n-1)/2$, then $\xi_2<1$, and so $\rho(S_{n,t})>n-2$.
		 Therefore, $\rho(H)\ge\rho(H_{n,m})>n-2$.
		If some component of $\x$, say $x_1$, is zero (where $x_1$ corresponds to the vertex $v_1$), then
\begin{align}
\rho(H)&=\x\s(H)\x^\top=(x_2,\ldots,x_n)\s(H-v_1)(x_2,\ldots,x_n)^\top\nonumber\\
&\le\rho(H-v_1)\le\rho(\overline K_{n-1})=n-2,\label{eq:rho(H)}
\end{align}
		which is a contradiction. So $\x$ has no zero components.
		
		(ii) Let $m=(n-1)/2$, then $t=(n-1)/2$ and $\xi=1$ is a zero of \eqref{eq:g}. This means that $\rho(H_{n,m})=n-2$. Now, the maximality of $\rho(H)$ implies that $\rho(H)\ge n-2$. On the other hand, as $\x$ has a zero component, by \eqref{eq:rho(H)}, $\rho(H)\le\rho(H-v_1)\le n-2$. Therefore, $\rho(H-v_1)= n-2$. Thus $H-v_1$ must be switching equivalent with $\overline K_{n-1}$. It follows that $H-v_1=K_{r,n-1-r}$ for some $0\le r\le (n-1)/2$. But, given that $n\ge4$, $r(n-1-r)>(n-1)/2=m$ unless $r=0$. Thus $H-v_1=\overline K_{n-1}$. It follows that $H=S_{n,(n-1)/2}$.
}\end{proof}

We will need the following lemma in the proof of our main result.

\begin{lemma}\label{lem:t<t'}
	Let $n\ge3$.
\begin{itemize}
  \item[\rm(i)] If $0\le t<(n-1)/2$, then $\rho(S_{n,t})$ has an eigenvector with all positive components and
  if $(n-1)/2\le t<n-1$, then $\rho(S_{n,t})$ has no eigenvector with all positive components.
  \item[\rm(ii)] If $1\le t<t'$ and $t+t'< n-1$, then $\rho(S_{n,t})>\rho(S_{n,t'})$. If $t+t'=n-1$, then $\rho(S_{n,t})=\rho(S_{n,t'})$.
\end{itemize}
\end{lemma}
\begin{proof}
(i)  Let $\rho=\rho(S_{n,t})$. For $t=0$, we have $\s(S_{n,0})=J-I$, so $\rho=n-1$ and $\1$, the all ones vector, is its eigenvector.

For $1\le t<n-1$, we know that $\rho$ is an eigenvalue of $Q$ given in \eqref{eq:Q}. By the properties of the equitable partitions,
  if $Q(x,y,z)^\top=\rho(x,y,z)^\top$, then $\x=(x,y,\ldots,y,z,\ldots,z)$
is an eigenvector of $\s(S_{n,t})$ for
  $\rho$, where in $\x$, the components $y$ and $z$ are repeated $t$ and $n-t-1$ times, respectively. Form
  $Q(x,y,z)^\top=\rho(x,y,z)^\top$ it follows that
\begin{eqnarray}
  -ty+(n-t-1)z &=& \rho x,\label{eq:x} \\
  -x+(t-1)y+(n-t-1)z &=&\rho y,\label{eq:y} \\
  x+ty+(n-t-2)z &=& \rho z.\label{eq:z}
\end{eqnarray}
By subtracting \eqref{eq:x} from \eqref{eq:y}, we obtain
\begin{equation}\label{eq:rho-2t+1}
(\rho-2t+1)y=(\rho-1)x.
\end{equation}

 If $1\le t<(n-1)/2$, then by Lemma \ref{lem:nonzero}\,(i), $\rho>n-2$ and
 $\x$ has no zero entry. So we may assume that $x>0$.
Since $0<\rho-2t+1\le\rho-1$, form \eqref{eq:rho-2t+1} it follows that $y\ge x$.
By subtracting \eqref{eq:y} from \eqref{eq:z}, we obtain $(\rho+1)(z-y)=2x$ which implies that $z>y$. So we have obtained
$z>y\ge x>0$ and so all the components of $\x$ are positive.

If $t=(n-1)/2$, then by Lemma~\ref{lem:nonzero}\,(ii),  $\rho=n-2$ and so $\rho-2t+1=0$. So by \eqref{eq:rho-2t+1},  $x=0$.

If  $(n-1)/2<t<n-1$,  as $t$ is an integer, we must have $t\ge n/2$.
So $\rho-2t+1<0$ and thus by \eqref{eq:rho-2t+1}, either of $x$ or $y$ are non-positive.

		 (ii) If $1\le t<t'\le(n-1)/2$, then $t(n-1-t)<t'(n-1-t')$. Let $\rho(S_{n,t})=n-1-\xi$ and $\rho(S_{n,t'})=n-1-\xi'$, for some $\xi,\xi'\le1$, according to Theorem \ref{thm:index}. Then $$\xi'(n-\xi')^2=4t'(n-1-t')>4t(n-1-t)=\xi(n-\xi)^2.$$ This is only possible if $\xi'>\xi$ from which the result follows.
		
		 If $t'>(n-1)/2$, then let $t'':=n-1-t'$. So $t''(n-1-t'')=t'(n-1-t')$. Thus by Theorem~\ref{thm:index}, $\rho(S_{n,t'})=\rho(S_{n,t''})$. On the other hand, as $t+t'<n-1$, we have $1\le t<t''\le(n-1)/2$, and so by the previous case, $\rho(S_{n,t})>\rho(S_{n,t''})=\rho(S_{n,t'})$.
\end{proof}

We recall that, for a graph $H$ of order $n$ and $\x=(x_1,\ldots,x_n)$, we have
$$\x A(H)\x^\top=2\sum_{ij\in E(H)} x_ix_j.$$

The next theorem is the main ingredient of the proof of our main result.

\begin{theorem}\label{thm:star}
	Let $H$ be a graph of order $n$ and size $m$ with $0<m<n-1$. Also, let $\x$ be a unit vector of length $n$ with all positive components. Then $\x \s(H)\x^\top\le\rho(S_{n,m})$.	The equality holds if and only if $m<(n-1)/2$, $H$ is isomorphic to $S_{n,m}$, and $\x$ is an eigenvector for $\rho(S_{n,m})$.
\end{theorem}
\begin{proof}
		 Let $V(H)=\{v_1,\ldots,v_n\}$, $\x=(x_1,\ldots,x_n)$ and $0<x_1\le\cdots\le x_n$ be the components of $\x$ with $x_i$ corresponding to $v_i$ for $i=1,\ldots,n$.
		 Let $i\ge2$ and $v_i$ be a non-isolated vertex of $H$.
		 If $v_i$ is not adjacent to $v_1$, it has some neighbor $v_j$ with $j\ge2$. Thus $x_1x_i\le x_jx_i$.
		 We replace the edge $v_iv_j$ by $v_1v_i$. If $H'$ is the resulting graph, then
		 \begin{align}
		 \x\s(H')\x^\top-\x\s(H)\x^\top&=\x\big(J-I-2A(H')\big)\x^\top-\x\big(J-I-2A(H)\big)\x^\top\nonumber\\
		 &=2\x A(H)\x^\top-2\x A(H')\x^\top\nonumber\\
		 &=4(x_ix_j-x_1x_i)\ge0.\label{eq:xixj-x1xi}
		 \end{align}
				 By applying this transformation on all non-isolated vertices $v_i$, we obtain a graph $H_0$ of order $n$, size $m$, and with  $\x\s(H_0)\x^\top\ge\x\s(H)\x^\top$ where in $H_0$, $v_1$ is adjacent to all non-isolated vertices. So $H_0$ has exactly one connected component $F_0$ of order at least $2$.
		If all the edges of $F_0$ are incident with $v_1$, then $F_0$ is already a star and we are done.
Otherwise, $E(F_0-v_1)\ne\emptyset$. Note that since $m<n-1$, the number of isolated vertices of $H_0$ is greater than $|E(F_0-v_1)|$. Now, if for some $v_iv_j\in E(F_0-v_1)$, there exists some isolated vertex $v_l$ of $H_0$ such that
$x_1x_l\le x_ix_j$, then we replace the edge $v_iv_j$ by $v_1v_l$.
So for the resulting graph, an  inequality similar to \eqref{eq:xixj-x1xi} holds.
We continue this process until no such a replacement is possible.
Let $H_1$ be the resulting graph and $F_1$ be its only non-trivial connected component. Then $H_1$ has the same order and size as $H$ does, and  $\x\s(H)\x^\top\le\x\s(H_1)\x^\top.$ If $E(F_1-v_1)=\emptyset$, then $F_1$ is a star and we are done.
So, assume that $E(F_1-v_1)\ne\emptyset$.
Also, we may assume that $V(F_1)$ consists of the first $r$ vertices $v_1,\ldots,v_r$, where $r\le n-2$.
(If this does not hold, there exist two vertices $v_i,v_j$ with $i>j$ such that $v_1v_i\in E(H_1)$ and $v_j$ is an isolated vertex of $H_1$.
Then we replace $v_1v_i$ by $v_1v_j$ and since $x_1x_j\le x_1x_i$,  we are done as above.)
Let $t$ be the smallest index such that $v_t$ is a non-isolated vertex in $F_1-v_1$.
Assume that $v_t$ has $k$ neighbors $v_{j_1},\ldots,v_{j_k}$.
By our assumption on $H_1$, we have
\begin{equation}\label{eq:x2xi}
x_tx_i<x_1x_{r+1},\quad\hbox{for}~i\in\{j_1,\ldots,j_k\}.
\end{equation}

We choose $k$ vertices outside $F_1$, namely $v_{r+1},\ldots,v_{r+k}$
and join $v_1$ to them and remove the edges $v_tv_{j_1},\ldots,v_tv_{j_k}$. Call the resulting graph $H_2$ and its non-trivial connected component $F_2$.
We define a new vector $\y=(y_1,y_2,\ldots,y_n)$ by the components:
\begin{align*}
y_1&=x_1-k\e,\\
y_t&=x_t+k\e,\\
 y_{j_1}&=x_{j_1}+\e, \ldots,y_{j_k}=x_{j_k}+\e,\\
 y_{r+1}&=x_{r+1}-\e,\ldots,y_{r+k}=x_{r+k}-\e,
\end{align*}
with $\e$ to be specified later and $y_i=x_i$ for the rest of the components.
For simplicity, we let
$$x:=x_1,~ a:=x_t,~ w:=\frac{x_{j_1}+\cdots+x_{j_k}}k,~ z:=\frac{x_{r+1}+\cdots+x_{r+k}}k.$$
Since the components of $\x$ are ascending, we have $z\ge w\ge a\ge x$. If $z=w$, then \eqref{eq:x2xi} fails. Hence, we have
\begin{equation}\label{eq:z>b>a>x}
z>  w\ge a\ge x.
\end{equation}
Also from \eqref{eq:x2xi}, we have $xz>wa$. From this and   \eqref{eq:z>b>a>x}, it turns out that
$$z+x> w+a.$$
By the way $\y$ is defined, the sum of the components of $\y$ coincides with that of $\x$, that is
$\x\1^\top=\y\1^\top$. This in turn implies that
\begin{equation}\label{eq:xJx=yJy}
\x J\x^\top=\y J\y^\top.
\end{equation}
We will specify $\e$ so that
\begin{align}
\|\y\|&=\|\x\|=1, \label{|x|=|y|}\\
 \y A(H_2)\y^\top&<\x  A(H_1)\x^\top.\label{yAy<xAx}
\end{align}
We have
$$\|\y\|^2-\|\x\|^2=2(k^2+k)\e^2+2(-kx+ka+kw-kz)\e.$$
So, to fulfill \eqref{|x|=|y|}, we set
\begin{equation}\label{epsilon}
\e:=\frac{x+z-a-w}{k+1}.
\end{equation}
It remains to show that with this choice of $\e$, \eqref{yAy<xAx} will be satisfied.

We first observe that
$$A(F_1)=\bordermatrix{~ & v_1 & v_t&W&U \cr
	v_1& 0 &1& \1&\1\cr
v_t&1&0&\1&{\bf0}\cr
W& \1^\top&\1^\top&B_1&C\cr
U&\1^\top&{\bf0}^\top&C^\top&B_2\cr
},\quad
A(F_2)=\bordermatrix{~ & v_1 & v_t&W&U &Z\cr
	v_1& 0 &1& \1&\1&\1\cr
	v_t&1&0&{\bf0}&{\bf0}&{\bf0}\cr
	W& \1^\top&{\bf0}^\top&B_1&C&O\cr
	U&\1^\top&{\bf0}^\top&C^\top&B_2&O\cr
	Z&\1^\top&{\bf0}^\top&O&O&O\cr
},$$
in which $W=\{v_{j_1},\ldots,v_{j_k}\}$, $U=V(F_1)\setminus\left(\{v_1,v_t\}\cup W\right)$, and $Z=\{v_{r+1},\ldots,v_{r+k}\}$.
Also we let $\w=(x_{j_1},\ldots,x_{j_k})$, $\z=(x_{r+1},\ldots,x_{r+k})$, and $\uu$ be the vector consisting of those components of $\x$ corresponding to the vertices in $U$.
Then
\begin{align*}
\x A(H_1)\x^\top&=(x,a,\w,\uu)A(F_1)(x,a,\w,\uu)^\top\\
&=2x(a+\1\w^\top+\1\uu^\top)+2a\1\w^\top+\w B_1\w^\top+2\w C\uu^\top+\uu B_2\uu^\top\\
&=2x(a+kw)+2akw+2x\1\uu^\top+\w B_1\w^\top+2\w C\uu^\top+\uu B_2\uu^\top,\\
\y A(H_2)\y^\top&=(x-k\e,a+k\e,\w+\e\1,\uu, \z-\e\1)A(F_2)(x-k\e,a+k\e,\w+\e\1,\uu, \z-\e\1)^\top\\
&=2(x-k\e)\left(a+k\e+\1(\w+\e\1)^\top+\1\uu^\top+\1(\z-\e\1)^\top\right)\\
& \quad+(\w+\e\1)B_1(\w+\e\1)^\top+2(\w+\e\1)C\uu^\top+\uu B_2\uu^\top\\
&=2(x-k\e)(a+k\e+kw+k\e+kz-k\e)+2(x-k\e)\1\uu^\top\\
& \quad+(\w+\e\1)B_1(\w+\e\1)^\top+2(\w+\e\1)C\uu^\top+\uu B_2\uu^\top.
\end{align*}
It follows that
\begin{equation}\label{eq:yA2y-xA1x}
\y A(H_2)\y^\top-\x  A(H_1)\x^\top=2(x-k\e)(a+kw+kz+k\e)-2x(a+kw)-2kaw+\delta_1+\delta_2,
\end{equation}
in which
\begin{align}
\delta_1&=(\w+\e\1)B_1(\w+\e\1)^\top-\w B_1\w^\top=2\e\w B_1\1^\top+\e^2\1B_1\1^\top\nonumber\\
&\le2\e\w(J_k-I_k)\1^\top+\e^2\1(J_k-I_k)\1^\top=2\e k(k-1)w+\e^2k(k-1),\label{eq:delta1}\\
\delta_2&=2(x-k\e)\1\uu^\top+2(\w+\e\1)C\uu^\top+\uu B_2\uu^\top-2x\1\uu^\top-2\w C\uu^\top-\uu B_2\uu^\top\nonumber\\
&=-2k\e\1\uu^\top+2\e\1C\uu^\top\le-2k\e\1\uu^\top+2\e\1J_k\uu^\top\le-2k\e\1\uu^\top+2\e k\1\uu^\top=0.\label{eq:delta2}
\end{align}
Combining \eqref{eq:yA2y-xA1x}, \eqref{eq:delta1}, and \eqref{eq:delta2} we obtain
$\y A(H_2)\y^\top-\x  A(H_1)\x^\top\le2kf$,
where
$$f=-\e(a+kw)+(x-k\e)(z+\e)-aw+\e(k-1)w+\e^2(k-1)/2.$$
By substituting \eqref{epsilon} and simplifying, we see that
$$2(k+1)f=2(a-z)(z-w)(k-1)+(w-a)^2+(x-z)(x+3z-2a-2w).$$
In the right side, by \eqref{eq:z>b>a>x} the first term is non-positive.
Also as $z>w$ and $x+z>a+w$, we have  $3z+x>3w+a$ implying that
$x+3z-2a-2w>w-a$. This together with $z-x> w-a$ implies that $(w-a)^2+(x-z)(x+3z-2a-2w)<0$.
Therefore, $f<0$ which establishes \eqref{yAy<xAx}.

Now, from \eqref{eq:xJx=yJy}, \eqref{|x|=|y|},  and  \eqref{yAy<xAx} it follows that
$$\x\s(H_1)\x^\top=\x(J-I)\x^\top-2\x A(H_1)\x^\top<\y(J-I)\y^\top-2\y A(H_2)\y^\top=\y\s(H_2)\y^\top\le\rho(H_2).$$
In $H_2$, $v_t$ is adjacent only to $v_1$ and $v_1$ has more neighbors than it does in $H_1$. We continue this process with other vertices which have edges not incident with $v_1$, and we replace such edges by new edges incident with $v_1$.
In all such steps, the index increases.
At the end, we will come up with a star $K_{1,m}$ with the center $v_1$ together with $n-m-1$ isolated vertices.

Finally, let the equality hold in the theorem.
Then in the above argument, $F_1$ must be an star. Since otherwise, $E(F_1-v_1)\ne\emptyset$, and then we obtain
the graph  $H_2$ such that $\x\s(H)\x^\top<\y\s(H_2)\y^\top\le\rho(S_{n,m})$ which means the equality is not possible.
Therefore, $H_1$ is already $S_{n,m}$ and $\x\s(H)\x^\top=\x\s(H_1)\x^\top=\rho(S_{n,m})$.
So $\x$ must be an eigenvector for $\rho(S_{n,m})$.  From the proof of Lemma \ref{lem:t<t'}\,(i), we see that
 $x_1<x_2=\cdots=x_{m+1}< x_{m+2}\cdots=x_n$ for $m>1$.
  It follows that if $E(H)\ne E(H_1)=E(S_{n,m})$, then $\x\s(H)\x^\top<\x\s(H_1)$. Therefore, $H=H_1=S_{n,m}$.
Since $\x$ has only positive components, by Lemma \ref{lem:t<t'}\,(i) we must have $m<(n-1)/2$.
\end{proof}

\begin{remark}\label{rem:NumNeg}\rm
If $\Gamma=(K_n,\sigma)$, then
 $$\x A(\Gamma)\x^\top=2\sum_{1\le i<j\le n}\sigma(ij)x_ix_j.$$
Suppose that $\x$ has no zero components and let $P,Q\subseteq\{1,\ldots,n\}$ be the sets of indices of the positive and negative components of $\x$, respectively. If $H$ is the unsigned subgraph induced by the negative edges of $\Gamma$, then the negative terms in $\x A(\Gamma)\x^\top$, i.e. the terms $\sigma(ij)x_ix_j<0$ correspond to  the edges $ij\in E(K_{P,Q})\Delta E(H)$, where $\Delta$ denotes the symmetric difference and $K_{P,Q}$ is the complete bipartite graph with parts $P$ and $Q$.
\end{remark}

Now, we are prepared to finish the proof.

\begin{proof}[\bf Proof of Theorem~\ref{thm:main}]
	For $n\le3$ there is nothing to prove, so we assume that $n\ge4$.
	Let $H$ be a graph with the largest index among  the graphs of order $n$ and size $m$. Assume that $\x$ is a unit eigenvector for $\rho(H)$.
	If $\x$ has a zero component, then by Lemma~\ref{lem:nonzero}, we have necessarily
	$m=(n-1)/2$ and $H$ is isomorphic with $S_{n,m}$, as desired.  Hence we suppose that $\x$ has no zero components.
	 We assume that $\x$ has some $s$ negative components, that is $$\x=(-x_1,\ldots,-x_s,x_{s+1},\ldots x_n),$$
	where all $x_i$'s are positive. We may assume that $s\le\lfloor n/2\rfloor$, otherwise we use $-\x$ instead of $\x$.
	 Let $\x'=(x_1,\ldots,x_n)$ with all positive components.
		Also let $K:=K_{\{1,\ldots,s\},\{s+1,\ldots,n\}}$ (if $s=0$, then $E(K)=\emptyset$).
	
  If $m=r(n-r)$ for some $0\le r\le\lfloor n/2\rfloor$, then we are done by Lemma~\ref{lem:index}.
 Hence suppose that $r(n-r)<m<(r+1)(n-r-1)$, for some $0\le r\le\lfloor n/2\rfloor-1$, and let
 $$a:=m-r(n-r), \quad b:=(r+1)(n-r-1)-m.$$
 Note that $a+b=n-2r-1\le n-1$.
By Remark \ref{rem:NumNeg}, the negative terms in $\x \s(H)\x^\top$ correspond to
 $$(E(H)\setminus E(K)) \cup (E(K)\setminus E(H)).$$

(i) If $s\le r$, then $|E(H)\setminus E(K)| \ge m-s(n-s)\ge a$.
 Let $H':=H\setminus E(K)$.   In $\x \s(H)\x^\top$ the terms corresponding to
 $E(H')$ are negative and in $\x' \s(H')\x'^\top$, the negative terms are exactly those corresponding to $E(H')$.
 This implies that $\x \s(H)\x^\top\le\x'\s(H')\x'^\top$.
  Now, let $H'_1$ be a subgraph of $H'$ with $a$ edges. Hence, $\x'\s(H')\x'^\top\le\x'\s(H'_1)\x'^\top$.
   But $\x'$ has positive components and $H'_1$ has less than $n-1$ edges. Thus by Theorem~\ref{thm:star},
   $\x'\s(H'_1)\x'^\top\le\rho(S_{n,a})$. So, we obtain $\rho(H)\le\rho(S_{n,a})$.

(ii)  If $s\ge r+1$, then $|E(K)\setminus E(H)| \ge s(n-s)-m\ge b$.
  Let $H'':=K\setminus E(H)$. The terms in $\x \s(H)\x^\top$ corresponding to
  $E(H'')$ are negative and in $\x' \s(H'')\x'^\top$, the negative terms are exactly those corresponding to $E(H'')$. It follows that $\x \s(H)\x^\top\le\x' \s(H'')\x'^\top$.
   Now, let $H''_1$ be a subgraph of $H''$ with $b$ edges. Hence, $\x'\s(H'')\x'^\top\le\x'\s(H''_1)\x'^\top$.
    But $\x'$ has positive components and $H''$ has less than $n-1$ edges. Thus  by Theorem~\ref{thm:star},
  $\x'\s(H''_1)\x'^\top\le\rho(S_{n,b})$. Therefore,  $\rho(H)\le\rho(S_{n,b})$.

 It follows that $\rho(H)\le\max\big(\rho(S_{n,a}),\rho(S_{n,b})\big)$.
 Let $t$ and $d$ be the parameters given in Definition \ref{def:Hm}.
 Note that $t=\min(a,b)$. Since $a+b\le n-1$, by Lemma~\ref{lem:t<t'}\,(ii),   $\max\big(\rho(S_{n,a}),\rho(S_{n,b})\big)=\rho(S_{n,t})$ which is equal to $\rho(H_{n,m})$ by Lemma \ref{lem:index} and we are done.

Now, suppose that the equality $\rho(H)=\rho(S_{n,t})$ holds.

If $t=a$, then we have $d=r$. On the other hand, by (i),
 $\x \s(H)\x^\top=\x'\s(H')\x'^\top=\x'\s(H'_1)\x'^\top=\rho(S_{n,a})$.
 The difference between $\x \s(H'_1)\x^\top$ and $\x'\s(H')\x'^\top$ is in the negative terms corresponding to $E(H')\setminus E(H'_1)$.
So $\x \s(H'_1)\x^\top=\x'\s(H')\x'^\top$ is possible only if $E(H')=E(H'_1)$.
 Thus $H'=H'_1$ which implies that $m-s(n-s)=a$ and so $s=r$. Also by the equality case in Theorem~\ref{thm:star},
$H'=H'_1$ must be isomorphic with $S_{n,a}$. Further, the negative terms in both $\x \s(H)\x^\top$ and $\x'\s(H')\x'^\top$ should coincide. It follows that $K$
 must be a subgraph of $H$ and $H\setminus E(K)$ is isomorphic to
$S_{n,t}$. Since $s=r=d$, it follows that $H$ is obtained by adding the edges of a star $K_{1,t}$ to $K_{d,n-d}$. As $m>d(n-d)$, this is the graph $H_{n,m}$.

If $t=b$,  then we have $d=r+1$ and $b<(n-1)/2$.
On the other hand by (ii),
$\x \s(H)\x^\top=\x'\s(H'')\x'^\top=\x'\s(H''_1)\x'^\top=\rho(S_{n,b})$.
The difference between $\x \s(H''_1)\x^\top$ and $\x'\s(H'')\x'^\top$ is in the negative terms corresponding to $E(H'')\setminus E(H''_1)$.
So $\x \s(H''_1)\x^\top=\x'\s(H'')\x'^\top$ is possible only if $E(H'')=E(H''_1)$.
Hence, $H''=H''_1$ which implies that $s(n-s)-m=b$ and so $s=r+1$.
By the equality case in Theorem~\ref{thm:star},
$H''=H''_1$ must be isomorphic with $S_{n,b}$. Further, the negative terms in both $\x \s(H)\x^\top$ and $\x'\s(H'')\x'^\top$ should coincide. It follows that $H$ must be a subgraph of $K$ and $K\setminus E(H)$ is isomorphic to
$S_{n,t}$. Since $s=r+1=d$, it follows that $H$ is obtained by removing the edges of a star $K_{1,t}$ from $K_{d,n-d}$. As $m< d(n-d)$, this is indeed the graph $H_{n,m}$.
 \end{proof}

\section*{Acknowledgements}
The first author carried this work during a Humboldt Research Fellowship at the University of Hamburg. He thanks the Alexander
von Humboldt-Stiftung for financial support.

\end{document}